\documentclass[twoside,12pt]{article}
\usepackage{mathrsfs}
\usepackage{amssymb}  
\usepackage{amstext}
\usepackage{amsmath,amscd,graphicx}
\usepackage{paralist}  
\usepackage{amsthm}  

\def\H{{\cal H}}
\def\L{\mathcal{L}}

\def\R{\mathbb{R}}

\def\H2{H^2(\R^N)}
\def\L2{L^2(\R^N)}

       \newtheorem{lemma}{\bf Lemma}
       \newtheorem{theorem}{\bf Theorem}[section]

       \newtheorem{remark}{\bf Remark}[section]
        \newtheorem{claim}{\bf Claim}[section]

       \numberwithin{equation}{section}

\begin{document}

\title{{\LARGE   On some model equations of Euler and Navier-Stokes  equations}
 \footnotetext{\small E-mail addresses: dudp954@nenu.edu.cn  }
 }

\author{{Dapeng Du}\\[2mm]
\small\it  School of Mathematics and Statistics, Northeast Normal University,\\
\small\it   Changchun 130024, P.R.China \\
 }

\date{}

\maketitle

\begin{quote}
\small \textbf{Abstract}:
We propose a two-dimensional generalization of Constantin-Lax-Majda model\cite{CLM}. Some results about singular solutions are given.
 This model might be the first step toward the singular solutions of the  Euler equations.

Along the same line (vorticity formulation), we present some further model equations.  They possibly models various aspects of difficulties related with the singular solutions of the Euler and Navier-Stokes equations.  We also make some discussions on the possible connection between turbulence and the singular solutions of the Navier-Stokes equations.



\indent \textbf{}


\end{quote}



\section{Introduction}

 The incompressible Navier-Stokes equations  describe the motion of incompressible viscous flow.   Whether singular solutions exist in 3D is one of the famous seven millennium prize problems. In terms of singular solutions, the study of incompressible  Euler equations looks most probably to be the first major step.  For the sake of convenience, we will omit {\it incompressible } from now on. 
The first nonlocal model was constructed by Constatin, Lax and Majda\cite{CLM}. 
They constructed a one-dimensional model and got the singular solution explicitly. The motivation is the vorticity formulation.  There are many  developments after this model (\cite{Sc, CCH, De} and others).

In this paper, we give some high dimensional generalizations of the Constantin-Lax-Majda model. The study of them might help the understanding of singular solutions to the Euler and Navier-Stokes equations. The motivation is still vorticity formulation.

  We first present a two-dimensional zero order scalar model. One may think of it as a nonlocal ODE. The good understanding of it possibly is among the first steps toward singular solutions of the Euler and Navier-Stokes equations.

 Then we give further models. In some sense, the vorticity formulation provides an explanation why the singular solutions to the Navier-Stokes equations are so hard. Roughly speaking, the zero order term is pro-singularity term. The first and second order terms are perturbations. To be able to construct original solution from vorticity, we need the initial vorticity to be divergence-free. Any of them is hard to handle. The combination of them composes one of the most difficult problems in mathematics.

 One potential major ppplication of singular solutions of the Navier-Stokes equations is about turbulence.
 There is a prevailing viewpoint that the turbulent theory is deeply connected with the solutions of the Navier-Stokes equations.  Actually one can interpret the behaviour of turbulence by the guessed properties of singular solutions to the Navier-Stokes equations.

 This paper is organized as follows.  In section two, we discuss the two-dimensional zero order scalar models.
 In section three, further models are given.  The possible connection between singular solutions of the Navier-Stokes equations and turbulence theory is presented in section four.  The notations we use are standard ones.

\section{  Zero order scalar models}
The vorticity formulation of the three-dimensional Euler equations is the following:

\begin{equation}\label{2.1}
w_{t}+u \cdot \nabla w-\nabla uw=0,
\end{equation}

In~$R^{3}$, the velocity~$u$~is given by the Biot-Savart law:
\begin{equation}\label{1.4}
u=\mbox{\rm curl}~\Delta^{-1}w.
\end{equation}
Note that the equations (\ref{2.1})  and (\ref{3.1}) are well defined when the vorticity is not divergence-free.
Define
\begin{equation}\label{2.2}
Z_{ij}=\partial _{ij}\Delta^{-1},x\in\Omega,
\end{equation}
where $\Omega$ is a bounded domain in $R^n$ or $R^n$ itself, $n \geq 2$.
If the domain is bounded, then the boundary condition for the $\triangle$ is the homogeneous Drichilet boundary condition:

\begin{equation}\label{2.3}
\Delta^{-1}w\mid_{\partial\Omega}=0.
\end{equation}
In $R^3$,  the term~$\nabla u$ can  be rewritten as
\begin{equation}\label{2.4}
\begin{split}\nabla u&=\nabla \mbox{\rm curl}\Delta^{-1}w\\
&=\left(
\begin{array}{ccc}
Z_{21}w_{3}-Z_{31}w_{2}&Z_{31}w_{1}-Z_{11}w_{3}&Z_{11}w_{2}-Z_{21}w_{1}\\
Z_{22}w_{3}-Z_{32}w_{2}&Z_{32}w_{1}-Z_{12}w_{3}&Z_{12}w_{2}-Z_{22}w_{1}\\
Z_{23}w_{3}-Z_{33}w_{3}&Z_{33}w_{3}-Z_{13}w_{3}&Z_{13}w_{2}-Z_{23}w_{1}\\
\end{array}
\right).
\end{split}
\end{equation}
\begin{remark}
If the domain has boundary, then in general, (\ref{1.4}) is not valid. Neither is \eqref{2.4}. But conceptually, the equality (\ref{2.4}) is still close to be true. We refer to \cite{Dur} and references therein for more details about the reconstruction of the velocity from the vorticity.
\end{remark}

The Constantin-Lax-Majda model has the following form:
\begin{equation}\label{1.1}
\theta_{t}=H(\theta)\theta,
\end{equation}

where $H$ is the Hilbert transform.

One nature generalization of Constantin-Lax-Majda model is:\\
{\bf Model 1.}
\begin{equation}\label{2.5}
w_{t}=Z_{11}w ~ w,x\in\Omega\subset R^{2}.
\end{equation}

\begin{claim}
The model equation is locally well-posed in $W^{1,p}, p>2$, $i.e. \ \forall w_{0}\in W^{1,p},\exists w\in C((0,T),W^{1,p}),s.t. \ w(x,0)=w_{0} $,\ $w$ satifies (\ref{2.5}).
\end{claim}
The proof is pretty standard and we omitted it.  

Next we present some elementary singular solutions of Model 1. One feature of the zero order model is that self-similar singular solutions could be considered in bounded domain.

Let
\begin{equation}\label{2.6}
w=\frac{1}{T-t}Q(\frac{x}{T-t}).
\end{equation}
Then the equation for $Q$ is
\begin{equation}\label{2.7}
Z_{11}Q~Q=Q.
\end{equation}
The interesting thing is that when the domain is an ellipse, the equation (\ref{2.7}) has constant solution.

\begin{theorem}\label{claim2.3}
Assume $\Omega=\{ax_{1}^{2}+bx_{1}x_{2}+cx_{2}^{2}<1\},a,c>0,b^{2}-4ac<0$, then the equation (\ref{2.7}) has a constant solution $Q=1+\frac{c}{a}$.
\end{theorem}

\begin{proof}
By the definition of ellipse, we know\\
\begin{equation*}
\Delta^{-1}~1=\frac{1}{2(a+c)}(ax_{1}^{2}+bx_{1}x_{2}+cx_{2}^{2}-1).
\end{equation*}
So
\begin{equation*}
Z_{11}~1=\frac{a}{a+c}.
\end{equation*}
Therefore $Q=1+\frac{c}{a}$ solves (\ref{2.7}). The theorem is proven.
\end{proof}

Going back to the original equation, we see that $w=\frac{a+c}{a} \cdot \frac{1}{T-t}$ is a singular solution to Model 1 if the domain is an ellipse. Theorem 2.1 seems to be the first singular solution result regarding two-dimensional nonlocal models.

Consider the following simpler version of Model 1.
\begin{equation}\label{a}
(Z_{11}+aZ_{22})w ~w=w.
\end{equation}
\begin{claim}
Assume $a>0$, the domain $\Omega$ is rectangle or whole space. Then for any measurable set $E\subset \Omega$, (\ref{a}) has solution $w\in L^2(\Omega)$ such that
\begin{equation*}
(Z_{11}+aZ_{22})w\big|_{\Omega \setminus E}=1.
\end{equation*}
\end{claim}
\begin{proof}
Without loss of generality, we may assume the rectangle is $(0,\pi) \times (0,\pi)$.
In this case,~$\sin k\cdot x,~k=(k_{1},k_{2})$,~$k_{i}$~positive integers are complete orthogonal basis and $w=\lambda_k \sin k \cdot x $.
So we have
\begin{equation*}
Z_{11}w=\sum_{k_{1},k_{2}=1}^{\infty}\frac{k_{1}^{2}}{k_{1}^{2}+k_{2}^{2}}\lambda_{k}.
\end{equation*}
Therefore
\begin{equation}\label{1.5}
\int_{(0,\pi)\times (0,\pi)}Z_{11} ww \mbox{\rm d}x\geq0.
\end{equation}

In whole space case,~similarly we have
\begin{equation}\label{1.6}
\int_{R^{2}}Z_{11} ww \mbox{\rm d}x=\int_{R^{2}}\frac{k_{1}^{2}}{k_{1}^{2}+k_{2}^{2}}|\tilde{w}|^{2}~\mbox{\rm d}k\geq0,
\end{equation}
where~$\widetilde{w}$~is the Fourier transform of~$w$. Also note that $Z_{11}$ is self-adjoint.

Define
\begin{equation*}
L_a=Z_{11}+aZ_{22}.
\end{equation*}
So under the assumptions of the current claim, $L_a$ is  coercive:
\begin{equation}
\langle L_a w,w\rangle_{L^2}\ge a\parallel w\parallel_{L^2}.
\end{equation}
The proof mainly comes from the coerciveness of $L_a$. Below is the details.

Note $$(\ref{a})\Longleftrightarrow (L_aw-1)w=0. $$
So the solving of (\ref{a}) reduces to find $w$ such that
\begin{equation*}
\begin{cases}
w=0, \quad x\in E,\\
L_aw=1,\quad x\notin E.
\end{cases}
\end{equation*}
Define
\begin{equation}\label{b}
\tilde L_a\tilde w=L_a w\big|_{\Omega\setminus E},
\end{equation}
\begin{equation*}
\quad w=\begin{cases}0, x\in E,\\
\tilde w, x\notin E,
\end{cases}
\tilde w\in L^2(\Omega\setminus E).
\end{equation*}

In some sense $\tilde L_a$ is the restriction of $L_a$ on $L^2(\Omega\setminus E)$. Note $\tilde L_a$ is also self-adjoint. Below we show $\tilde L_a$ is one-to-one and the inverse of it is bounded.\\
1) one-to-one. \\
Assume $\tilde w\in L^2(\Omega\setminus E)$ and $\tilde L_a \tilde w=0$. Define
 \begin{equation*}
w=\begin{cases}
0,\quad x\in E,\\
\tilde w,\quad x\in \Omega\setminus E
\end{cases}
\mbox{as in (\ref{b})},
\end{equation*}
 Then we have
  \begin{equation}\label{c}
  \langle L_a w, w\rangle_{L^2(\Omega)}=\langle \tilde L_a\tilde w, \tilde w\rangle_{L^2(\Omega\setminus E)}=0.
  \end{equation}
 Therefore
   \begin{equation*}
    w\equiv 0,
      \end{equation*}
      which implies $\tilde w\equiv 0$. So $\tilde L_a$ is one-to-one.\\
      2) The inverse is bounded.\\
      Using (\ref{b}) and (\ref{c}), we get
      \begin{eqnarray*}
      \langle\tilde L_a\tilde w, \tilde w\rangle&=&\langle L_a w, w\rangle\nonumber\\&\ge& a\parallel w\parallel_{L^2}^2\nonumber\\&\ge &a\parallel \tilde w\parallel_{L^2}.
      \end{eqnarray*}
      So
      \begin{eqnarray*}
      \parallel\tilde w\parallel^2\le a^{-1}\parallel\tilde w\parallel_{L^2}\parallel\tilde L_a\tilde w\parallel_{L^2}.
      \end{eqnarray*}
      ie.
      \begin{eqnarray*}
      \parallel\tilde w\parallel_{L^2}\le a^{-1}\parallel\tilde L_a\tilde w\parallel_{L^2}.
      \end{eqnarray*}
      Hence the inverse of $\tilde L_a$ is bounded. \\
     Next we show $\tilde L_a$ is onto.\\
      Assume the contrary. Then $\exists\tilde w\notin \tilde L_a(L^2(\Omega\backslash E))$. Since the inverse of $\tilde L_a$ is bounded, the space $\tilde L_a(L^2(\Omega\setminus E))$ is closed. Denote $\tilde w_1$ the projection of $\tilde w$ on $\tilde L_a(L^2(\Omega\setminus E))$. Then
      $$\langle \tilde w-\tilde w_1, \tilde L_a(L^2(\Omega\setminus E))\rangle=0$$
      So $\forall \tilde u \in L^2(\Omega\setminus E)$,
      $$\langle \tilde L_a(\tilde w-\tilde w_1),\tilde u \rangle=0.$$
      This means $\tilde L_a(\tilde w-\tilde w_1)=0$.
      So $\tilde w=\tilde w_1$. A  contradiction.
      Therefore $\tilde L_a$ is onto.\\
      After proving $\tilde L_a$ is onto, the proof is essential finished. Let $\tilde w=\tilde L_a^{-1}1$ and \begin{eqnarray*}w=\begin{cases}0,\quad x\in E,\\
      \tilde w, \quad x\notin E. \end{cases}\end{eqnarray*}
      Then $w$ is a solution to (\ref{a}). The claim is proven.
\end {proof}
\begin{remark}
 The claim above might help a little bit in the study of singular solutions of Model 1.
      \end{remark}
There is a small generalization of Model 1. \\
{\bf Model $1^{'}$}\\
\begin{equation}\label{2.8}
w_{t}=Z_{12}w~w,x\in\Omega\subset R^{2}.
\end{equation}

Model $1^{'}$ seems a little bit harder than Model 1. There are some related evidences. For instance, assume $\varphi\in L^{2}(R^{2})$,~then

$$\int_{R^{2}}Z_{12}\varphi\varphi \mbox{\rm d}x=\int_{R^{2}}\frac{k_{1}k_{2}}{k_{1}^{2}+k_{2}^{2}}|\tilde{\varphi}|^{2} \mbox{\rm dk} \quad
\mbox{\rm will change sign}  $$
where $\tilde{\varphi}$ is the Fourier transform of $\varphi$.
Also for the simple singular solution in the ellipse, we need $b$ to be  non-zero.
More precisely,~we have the following theorem.
\begin{theorem}
Assume $\Omega=\{ax_{1}^{2}+bx_{1}x_{2}+cx_{2}^{2}<1\},a,c>0,b^{2}-4ac<0,b\neq0$,~then
\begin{equation*}
w(t)=\frac{1}{T-t}\cdot\frac{2a+2c}{b}
\end{equation*}
is a self-similar singular solution to (\ref{2.8}).
\end{theorem}
\begin{proof}
The proof is essentially the same as Theorem 2.1.
\end{proof}
 For zero order models, bounded domain case might be simpler than the whole space case since the former is compact region.

\section{Further models}
The vorticity formulation of the three-dimensional Navier-Stokes equations is the following:
\begin{equation}\label{3.1}
w_{t}-\Delta w+u\cdot\nabla w-\nabla u w=0.
\end{equation}
By simplifying the zero order term, removing first order or second order term, we could get various model equations of the Navier-Stokes equations. The usual way of simplifying zero order term is to  replace $\nabla u$ with simpler zero order operater.

{\bf Examples.}
\begin{equation}\label{3.2}
w_{t}=
\left(
\begin{array}{cc}
w_{1}+Z_{11}w_{1}&\frac{1}{2}w_{1}\\
\frac{1}{2}w_{1}&w_{1}+Z_{11}w_{1}\\
\end{array}
\right)w,
~x\in\Omega\subset R^{2},w\in R^{2},
\end{equation}

\begin{equation}\label{3.3}
w_{t}=\nabla u~w,~x\in\Omega\subset R^{3},w\in R^{3},
\end{equation}

\begin{equation}\label{3.4}
w_{t}+u\cdot\nabla w-Z_{11}w~w=0,~x\in R^{2},
\end{equation}

\begin{equation}\label{3.5}
w_{t}-\Delta w-Z_{11}w~w=0,~x\in R^{2},
\end{equation}

\begin{equation}\label{3.6}
w_{t}-\Delta w+u\cdot\nabla w-Z_{11}w~w=0,~x\in R^{2},
\end{equation}

In 2D,~$u=(-\partial_{x_{2}}\Delta^{-1}w,\partial_{x_{1}}\Delta^{-1}w)$.

Roughly speaking, equation (\ref{3.2}) is a simple situation that the model is a system (The matrix in (\ref{3.2}) is symmetric, and the equation also has simple singular solutions similar with Theorem 2.1). Equation (\ref{3.3})  models the pro-ingularity effect of zero order term in the vorticity formulation. Equations (\ref{3.4}), (\ref{3.5}) and (\ref{3.6}) model the effects of first order perturbation, second order perturbation, first and second order perturbation combined for scalar equations.

\begin{remark}
Different from zero order model, it seems that for first and second order models, the whole space case is inclined to be first considered. One reason is that self-similar singular solutions for PDEs only occur in whole space.
\end{remark}
Next we make some further discussions.

{\bf .Divergence-free requirement on initial vorticity} First note that if we have solutions to the vorticity equations in 3D and the initial vorticity is divergence free, then we can construct solutions to the original equations.
\begin{lemma}
Assume $w$ satisfy  (2.1),(2.2)~or~(3.1),(2.2) in $R^{3}$,~$\mbox{\rm div}~w(x,0)=0$,~$w$~are regular enough and have good decay at infinity. Then $u=\mbox{\rm curl}~\Delta^{-1}w$ solves the original Euler or Navier-Stokes equations.
\end{lemma}
\begin{proof}
We will mainly present the proof for the Euler equations. The situation for the Navier-Stokes equations is very similar. Let $w\in C((0,T),~H^{2}(R^{^{3}}))$.~Let $u_{0}=\mbox{\rm curl}~\Delta^{-1}w(x,0)$,~then $u_{0}\in W^{2,6}$.~Standard local existence for Euler equations (for instance,\cite{MB}) implies there exists a solution $u^{(1)}\in C((0,T_{1}),W^{2,6}),~T_{1}=T_{1}(\parallel u_{0}\parallel_{W^{2,6}})$.~Set $w^{(1)}=\mbox{\rm curl}~u^{(1)}$.~Note that $\mbox{\rm curl}~\Delta^{-1}w$ is the unique solution to\\
\begin{equation}\label{3.15}
\begin{cases}
\mbox{\rm curl}~v=w_{1},~~x\in R^{3},\\
\mbox{\rm div}~v=0. \\
\end{cases}
\end{equation}
Then $w^{(1)}$ satisfies (2.1)-(2.2) and $w^{(1)}(x,0)=w(x,0)$.~Take $w^{(2)}=w^{(1)}-w$.~Standard energy estimates imply $w^{(2)}\equiv 0$.~Now the Euler equations case is finished by standard continuation argument.~For the Navier-Stokes equations,~we can assume $w\in C((0,T),L^{2}(R^{3}))$. The rest is essentially the same.
\end{proof}
Note that the singular solutions with well-posed initial data are regular before the singularities occur.  So Lemma 3.1 implies the singular solutions to the vorticity equations will generate singular solutions to the original equations. Therefore as long as a model is system, there are two situations: the initial data is divergence-free or not. In general, the divergence-free requirement would make situation harder.

If the dimension is higher than three, it is convenient to think the velocity as 1-form and vorticity as 2-form. In this case, the divergence-free requirement becomes $\mbox {\rm d} w_0=0$,  where d is the exterior differential and $ w_0$ is initial vorticity. We refer to \cite{Se} for more details in the case of $R^n, n>3$.

{\bf . Skew-symmetry of zero order term}
In the roughest sense, one may think of the zero order term $\nabla u w$ as $w^2, w \in R$. Therefore one might expect that it has some pro-singularity effect.
\begin{claim}
Singular solutions generated from constant don't hold true for equation (\ref{3.3}).
\end{claim}
\begin{proof}
Define the generalized Kronecker sign:
\begin{equation*}
\delta_{jl}^{i}=
\begin{cases}
~1,~~(i,j,l)~is~an~even~arrangement~of~(1,2,3), \\
-1,~(i,j,l)~is~an~odd~arrangement~of~(1,2,3),\\
~0,~~otherwise. \\
\end{cases}
\end{equation*}

So
\begin{equation*}
\begin{split}(\nabla u)_{mi}&=\partial_{m}u_{i}\\
&= \partial_{m}\delta_{jl}^{i}\partial_{j}\Delta^{-1}w_{l}\\
&=\delta_{jl}^{i}Z_{jm}w_{l}.
\end{split}
\end{equation*}
And
\begin{equation}\label{3.7}
(\nabla u~w)_{m}=\delta_{jl}^{i}Z_{jm}w_{l}w_{i}.
\end{equation}
Given any constant vector $c\in R^{3}$,~$Z_{jm}~c=a_{jm}c$. Here the domain is ~$\{a_{ij}x_{i}x_{j}<1\},~a_{ij}=a_{ji},~\sum\limits_{i=1}^{3} a_{ii}=1$, and $(a_{ij})$~is positive definite.
Therefore
\begin{equation}\label{3.8}
\begin{split}&(\nabla uw)_{m}\mid_{w=c}\\
=&\delta_{jl}^{i}Z_{jm}c_{l}c_{i}\\
=&\delta_{jl}^{i}a_{jm}c_{l}c_{i}\\
=&0.
\end{split}
\end{equation}
The claim is proven.
\end{proof}
The claim above suggests that the zero order term has certain algebraic skew-symmetry, which may cause some more trouble in the study of singular solutions.

{\bf .Possible steps toward  Euler equations}

  In the luckiest scenario, the study of model equations might lead to the existence of singular solutions of the Euler equations and even Navier-Stokes equations.  The following are possible steps toward Euler equations.

\begin{enumerate}

 \item Model 1,
 \item \eqref{3.3},
 \item \eqref{3.4},
 \item The whole Euler equations.
 
\end{enumerate}

\begin{remark}
It was suggested in \cite[~p.3]{Du05} that the degree of difficulty for singular solutions to Navier-Stokes equations may decrease a lot in higher dimensions. Probably this scenario will also hold true for certain second order models.
\end{remark}

\begin{remark}
 For zero order models, if there are no divergence-free requirements on the initial data, the self-similar singular solutions probably exist.
 But for more complicated situations, one might have to work on singular solutions with general form.
 One evidence is that the Navier-Stokes equations don't have self-similar singular solutions at any dimensions\cite{NRS,Ts}. There were also no reliable numerical evidence that Euler equations have self-similar singular solutions.
\end{remark}

\section{Possible connection with turbulence}
It is well accepted that the main features of turbulence is irregular, random, and chaotic. Based on what's known on Navier-Stokes equations and the features of turbulence, it seems reasonable to make the following guess.

\noindent {\bf Conjecture 4.1. } {\it The singular solutions of three-dimensional Navier-Stokes equations generically are fluctuated.}

Using the conjecture above, we could interpreter the turbulence in the following way. Since the solution is fluctuated singular, the average of it is irregular. The randomness comes from the infinite amplifying effect of fluctuated singular solution over arbitrarily small experimental error. The chaotic behavior could be explained in the similar way.

\begin{remark}
 The difficulties for singular solutions of the Navier-Stokes equations might be viewed as the combination of difficulties for the local convection-diffusion equations with energy conservation and Euler equations. 
The results on model equations\cite{DL, PS} and numerical simulation for Euler equations suggest that,  in dimension five and higher the usual singular solutions possibly are also typical for Navier-Stokes equations. There is no information in dimension four so far.
\end{remark}

At this stage little is known regarding the singular solution of the Navier-Stokes equations. Therefore the application in the turbulence theory is not much. With the development of the mathematical theory on the singular solutions, more and more applications could be expected. To some degree, the good understanding of turbulence may depend on the good understanding of singular solutions to the three-dimensional Navier-Stokes equations.

\section*{Acknowledgements}
The author wishs to thank Hongjie Dong and Vladimir Sverak for valuable comments. The author also would like to thank Yipeng Shi for wonderful discussions on turbulence. 
 The author was partially supported by NSFC grant No. 11571066.

\bibliographystyle{amsplain}

\end{document}